\documentclass[abstracton]{scrartcl}

\usepackage[T1]{fontenc}
\usepackage[utf8]{inputenc}

\usepackage{amsmath}
\usepackage{amssymb}
\usepackage{amsthm}
\usepackage[backend=biber,hyperref=true,backref=true,style=alphabetic,maxnames=99]{biblatex}
\usepackage{booktabs}
\usepackage{csquotes}
\usepackage{enumitem}
\usepackage{mathrsfs}
\usepackage{mathtools}
\usepackage{microtype}
\usepackage{nth}

\usepackage{bm}
\usepackage{comment}
\usepackage{hyperref}
\usepackage{mathtools}
\usepackage{tgpagella}
\usepackage{tgheros}
\usepackage{eulervm}
\usepackage[textsize=small,colorinlistoftodos,disable]{todonotes} 

\numberwithin{equation}{section}

\DeclareFieldFormat[online]{title}{\mkbibquote{#1\isdot}}
\DeclareFieldFormat[online]{date}{(#1)}
\renewbibmacro{in:}{}
\DeclareFieldFormat[article]{pages}{#1\isdot}

\newcommand{\CC} {\mathbb{C}}

\newcommand{\Z} {\mathbb{Z}}
\newcommand{\Q} {\mathbb{Q}}

\newcommand{\PP}{\mathbb{P}}
\renewcommand{\P} {\mathbb{P}}

\newcommand{\normaliz}[1]{\widetilde{#1}}
\newcommand{\CoverCurve}[1]{#1'}

\DeclareMathOperator{\Ext}{Ext}

\DeclareMathOperator{\Hom}{Hom}

\DeclareMathOperator{\Nm}{Nm}
\DeclareMathOperator{\ord}{ord}

\DeclareMathOperator{\Pic}{Pic}

\DeclareMathOperator{\Sym}{Sym}
\DeclareMathOperator{\Supp}{Supp}

\DeclareMathOperator{\rk}{rk}

\newtheorem{proposition}{Proposition}[section]
\newtheorem{theorem}[proposition]{Theorem}
\newtheorem*{theorem*}{Theorem}
\newtheorem{corollary}[proposition]{Corollary}
\newtheorem*{corollary*}{Corollary}

\newtheorem*{conjecture*}{Conjecture}
\newtheorem{lemma}[proposition]{Lemma}
\newtheorem*{lemma*}{Lemma}

\theoremstyle{definition}
\newtheorem{definition}[proposition]{Definition}
\newtheorem*{definition*}{Definition}
\newtheorem{example}[proposition]{Example}
\newtheorem*{example*}{Example}
\newtheorem{remark}[proposition]{Remark}
\newtheorem*{remark*}{Remark}

\newtheorem*{question*}{Question}

\newtheorem*{result*}{Result}

\newcommand{\MgSymbol}{\mathcal{M}}

\newcommand{\RgSymbol}{\mathcal{R}}
\newcommand{\RgStackSymbol}{\mathsf{R}}
\newcommand{\AgSymbol}{\mathcal{A}}
\newcommand{\UnivGrdSymbol}{\mathfrak{G}}

\newcommand{\UnivCrdSymbol}{\mathfrak{C}}

\newcommand{\Mg}[1]{\MgSymbol_{#1}}

\newcommand{\Ag}[1]{\AgSymbol_{#1}}
\newcommand{\Rg}[1]{\RgSymbol_{#1}}

\newcommand{\MgBar}[1]{\overline{\MgSymbol}_{#1}}

\newcommand{\RgBar}[1]{\overline{\RgSymbol}_{#1}}

\newcommand{\RgBarStack}[1]{\overline{\RgStackSymbol}_{#1}}
\newcommand{\Rootgl}[1]{\operatorname{Root}_{#1}}

\newcommand{\BoundaryPCZero}{\Delta^0_0}

\newcommand{\RgPCZero}[1]{\RgSymbol^0_{#1}}

\newcommand{\RgBarDesing}[1]{\widehat{\RgSymbol}_{#1}}

\newcommand{\UnivGrd}[2]{\UnivGrdSymbol^{#1}_{#2}}

\newcommand{\UnivGrdTors}[3]{\UnivGrdSymbol^{#1,(#3)}_{#2}}

\newcommand{\UnivCrdTors}[3]{\UnivCrdSymbol^{#1,(#3)}_{#2}}

\newcommand{\PrymPoincare}{\mathscr{P}}

\newcommand{\PoincareBundle}{\mathscr{L}}
\newcommand{\RgForgetfulMap}{\pi}

\newcommand{\UnivCrdTorsMorph}{\chi}
\newcommand{\UnivGrdTorsMorph}{\sigma}

\newcommand{\MukaiDivSymbol}{\mathcal{B}}
\newcommand{\MukaiDiv}[2]{\MukaiDivSymbol_{#1,#2}} 
\newcommand{\MukaiDivBar}[2]{\overline{\MukaiDivSymbol}_{#1,#2}}
\newcommand{\MukaiMg}[1]{\MgSymbol_{#1}^\mu}
\newcommand{\MukaiRg}[1]{\RgSymbol_{#1}^\mu}

\newcommand{\Addresses}{{
  \bigskip
  \footnotesize
  \textit{E-mail address}:  \texttt{math@gregorbruns.eu}

}}

\newcommand{\SerreDual}[1]{\omega_C \otimes #1^{-1}}
\newcommand{\RgPrime}[1]{\RgSymbol^\prime_{#1}}
\newcommand{\virt}{\mathrm{virt}}
\newcommand{\DegenMorph}{\phi}

\excludecomment{dummyversion}

\title{Twists of Mukai bundles and the geometry of the level $\bm{{3}}$ modular variety over $\bm{{\MgBar{8}}}$}
\author{Gregor Bruns}
\date{}

\begin{document}

\maketitle

\begin{abstract}
  For a curve $C$ of genus $6$ or $8$ and a torsion bundle $\eta$ of order $\ell$
  we study the vanishing of the space of global sections of the twist $E_C \otimes \eta$
  of the rank two Mukai bundle $E_C$ of $C$.
  The bundle $E_C$ was used in a well-known construction of
  Mukai which exhibits general canonical curves of low genus
  as sections of Grassmannians in the Plücker embedding.

  Globalizing the vanishing condition, we obtain divisors on the moduli spaces
  $\RgBar{6,\ell}$ and $\RgBar{8,\ell}$ of pairs $[C, \eta]$.
  First we characterize these divisors
  by different conditions on linear series on the level curves,
  afterwards we calculate the divisor classes.
  As an application, we are able to prove that $\RgBar{8,3}$ is of general type.
\end{abstract}

\section{Introduction}

It is a famous result of Mukai (see \autocite{Mukai1993}) that a general canonical genus $8$ curve
is a linear section of the $8$-dimensional Grassmannian $G(2,6)$ in $\P^{14}$.
In a similar fashion, the general genus $6$ curve is the complete intersection
of a $4$-dimensional quadric and the $6$-dimensional Grassmannian $G(2,5)$ in $\P^{9}$.%

In both cases, the maps from the curve $C$ to the Grassmannian are induced by the global sections
of an (up to isomorphism) uniquely determined stable rank $2$ bundle $E_C$ with canonical determinant,
which we call the \emph{Mukai bundle} of $C$.
We have $h^0(C, E_C) = 5$ in genus $6$ and $h^0(C, E_C) = 6$ in genus $8$.
Since the bundle $E_C$ captures the geometry of $C$, it is a natural problem to study
loci of curves where $E_C$ shows non-generic behaviour.  In particular, we are
interested in divisorial conditions involving $E_C$ on moduli spaces of curves.

We let $\ell$ be a prime number, $C$ a general curve of genus $g = 6$ or $g = 8$
and $\eta \in \Pic^0(C)[\ell]$ a line bundle of order $\ell$.
Then we can consider the twisted bundle $E_C \otimes \eta$ and in particular its
space of global sections $H^0(C, E_C \otimes \eta)$.  Since the slope
of $E_C$ is $g - 1$, we expect $H^0(C, E_C \otimes \eta) = 0$ and the locus
\begin{equation*}
  \left\{ [C] \in \Mg{g} ~\left|~
      H^0(E_C \otimes \eta) \not= 0\text{ for some } \eta \in \Pic^0(C)[\ell]
    \setminus \{ \mathcal{O}_C \} \right.\right\}
\end{equation*}
to be a divisor in $\Mg{g}$, the moduli space of curves.
In fact it is more natural to study the question on the
modular variety $\Rg{g,\ell}$ parametrizing pairs $[C, \eta]$
of smooth curves of genus $g$ together with a nontrivial $\ell$-torsion line bundle.  These spaces
have been constructed and compactified in \autocite{CEFS2013}.  On $\Rg{g,\ell}$ we define the locus
\begin{equation*}
  \MukaiDiv{g}{\ell} = \left\{ [C, \eta] \in \Rg{g,\ell}
    ~\left|~ H^0(C, E_C\otimes \eta) \not= 0 \right.\right\}
\end{equation*}
which is of codimension at most one in $\Rg{g,\ell}$ and expected to be a divisor.  This we prove:
\begin{theorem}
  \label{thm:bgl-are-divisors}
  In both $g = 6$ and $g = 8$ and for every prime $\ell$ the locus $\MukaiDiv{g}{\ell}$
  is a divisor in $\Rg{g,\ell}$.
\end{theorem}
One of the primary motivations to study pairs $[C, \eta]$ originates from
the study of \emph{Prym varieties}.
Recall that, for $\ell = 2$, such a pair $[C, \eta] \in \Rg{g,2}$ corresponds to
an étale double cover $\pi\colon \CoverCurve{C} \rightarrow C$.
The Prym variety $\Pr(C, \eta)$ associated to $\pi$ is an abelian variety of dimension $g - 1$
which we construct by considering the \emph{Norm map}
\begin{equation*}
  \Nm_\pi\colon \Pic^{2g - 2}(\CoverCurve{C}) \rightarrow \Pic^{2g - 2}(C),
  \quad
  \mathcal{O}_{\CoverCurve{C}}(D) \mapsto \mathcal{O}_C(\pi_\ast D)
\end{equation*}
and then letting
\begin{equation*}
  \Pr(C, \eta) = \Nm_\pi^{-1}(K_C)^\mathrm{+} =
  \big\{
  L \in \Nm_\pi^{-1}(K_C) ~\big|~ h^0(C, L) \equiv 0 \pmod{2}
  \big\}
\end{equation*}
One can show that $\Pr(C, \eta)$ is principally polarized.
We then get a morphism
\begin{equation*}
  \Rg{g,2} \rightarrow \Ag{g-1}, \quad [C,\eta] \mapsto \Pr(C, \eta)
\end{equation*}
called the \emph{Prym map}, to the moduli space $\Ag{g-1}$
of principally polarized abelian varieties of dimension $g - 1$.
Prym varieties play an important role in the study of $\Ag{g}$
since a general abelian variety of dimension at most $5$ is a Prym.
On the other hand, recall that the general abelian variety of dimension at
least $4$ is not the Jacobian of a curve.
Hence Prym varieties make the study of abelian varieties amenable to techniques from curve theory
in a larger range than by just studying Jacobians.

For $\ell \geq 3$ one can analogously assign a cyclic unramified cover
$\CoverCurve{C} \rightarrow C$ of degree $\ell$ to any pair
$[C, \eta] \in \Rg{g,\ell}$.  However, this process is only
reversible if we consider such covers together with a generator of their Galois group.

Our main interest lies in furthering the understanding of the birational geometry of
the spaces $\Rg{g,\ell}$.
It is known (see \autocite{FL2010} and \autocite{Bruns2016})
that $\Rg{g,2}$ is of general type for $g \geq 14$ while $\Rg{g,3}$ is known to be of
general type for $g \geq 12$ (see \autocite{CEFS2013}).  We can use the divisor
$\MukaiDiv{8}{3}$ previously constructed to prove the following:
\begin{theorem}
  \label{thm:main-theorem}
  $\Rg{8,3}$ is of general type.
\end{theorem}
Note that we now have a result for genus $8$ and level $3$ while there is currently
nothing known about $\Rg{9,3}$ and $\Rg{10,3}$.  The Kodaira dimension of $\Rg{11,3}$
is at least $19$ (proved in \autocite{CEFS2013}) but our theorem actually suggests
that all three spaces should be of general type.

We recall that Mukai bundles also exist in genus $7$ and $9$.
Hence one can hope to make use of them in order to exhibit
divisors similar to $\MukaiDiv{8}{3}$.  The difficulty is that in genus $7$ the
Mukai bundle $E_C$ has rank $5$ and in genus $9$ it is of rank $3$, making its
slope a non-integer rational number in both cases.
An approach using vanishing conditions of global
sections of a twist $E_C \otimes \eta$ is therefore not going to work.
Possibly one can use, e.g. for $g = 9$, the bundle $\Sym^3 E_C$
which has integral slope $2g - 2$.

We also note that, although Theorem \ref{thm:bgl-are-divisors} includes genus $6$,
the resulting divisor $\MukaiDiv{6}{3}$ does not enable us to prove a statement similar to
Theorem \ref{thm:main-theorem}.
On the other hand, $\Rg{g,3}$ is known to be unirational for $g \leq 5$
(\autocite{BC2010,BV2010,VerraR53}).  This makes the study of $\Rg{6,3}$ especially
interesting, since it is likely to be a transitional case.

We now explain the strategy of the proof of Theorem \ref{thm:main-theorem}.
The method of obtaining general type results is by constructing divisors with a divisor class
in $\Pic_{\Q}(\RgBar{g,\ell})$ satisfying certain numerical bounds.
Here $\RgBar{g,\ell}$ is the modular compactification obtained by using
quasi-stable level $\ell$ curves as described in \autocite{CEFS2013}.
Hence the first step is to calculate the divisor class of the closure $\MukaiDivBar{g}{\ell}$
in the compactification $\RgBar{g,\ell}$.  We do this for both $g = 6$ and $8$ and for all $\ell$.

In fact it is enough to calculate the class on an appropriate partial compactification $\RgPrime{g,\ell}$ of $\Rg{g,\ell}$
containing only smooth and irreducible one-nodal curves.  On a cover of $\RgPrime{g,\ell}$ we express
the closure of $\MukaiDiv{g}{\ell}$ as the degeneracy locus of a morphism $\DegenMorph_{g,\ell}$ between
vector bundles of the same rank.  Using Porteous' formula and the machinery
for calculating Chern classes of vector bundles over $\MgBar{g}$, developed in \autocite{Farkas2009},
we then show:
\begin{theorem}
  We have the following expressions for the pushforward to $\RgPrime{g,\ell}$ of the
  classes of the degeneracy loci of $\DegenMorph_{g,\ell}$:
  \begin{enumerate}[label=\alph*)]
  \item The virtual class of the closure of $\MukaiDiv{6}{\ell}$ in $\RgPrime{6,\ell}$ is given by
    \begin{equation*}
      [\MukaiDivBar{6}{\ell}]^\virt = 35 \lambda - 5 (\delta_0' + \delta_0'')
      - \frac{5}{\ell} \sum_{a = 1}^{\lfloor \ell/2 \rfloor}( \ell^2 - a\ell + a^2 )\delta_0^{(a)}
    \end{equation*}
  \item The virtual class of the closure of $\MukaiDiv{8}{\ell}$ in $\RgPrime{8,\ell}$ is given by
    \begin{equation*}
      [\MukaiDivBar{8}{\ell}]^\virt = 196 \lambda - 28 (\delta_0' + \delta_0'')
      - \frac{14}{\ell} \sum_{a = 1}^{\lfloor \ell/2 \rfloor} (2\ell^2 - a\ell + a^2) \delta_0^{(a)}
    \end{equation*}
  \end{enumerate}
  In particular, the classes $[\MukaiDivBar{g}{\ell}]^\virt - n[\MukaiDivBar{g}{\ell}]$
  are effective and entirely supported on the boundary of $\RgPrime{g,\ell}$ for some $n \geq 1$.
\end{theorem}
By describing the degeneracy of the morphism used in Porteous' formula along
the boundary we can improve these divisor classes still further.
Similarly, we also improve a divisor class found in \autocite{CEFS2013}.
Combining these results we can prove our Main Theorem \ref{thm:main-theorem}.

\paragraph{Outline of the paper}

We begin in section \ref{sec:moduli-spaces} by recalling some facts
about the moduli spaces $\RgBar{g,\ell}$.
This is followed in section \ref{sec:mukai-bundles} by a review of the results
we need about Mukai bundles.
Then, in section \ref{sec:divisors}, we discuss the loci
$\MukaiDiv{g}{\ell}$ and show that they are divisors.
To do this, we prove a more general statement about vanishing
of global sections of twists of semistable vector bundles
under the right hypotheses (see Theorem \ref{thm:vector-bundle-vanishing}).
We also reinterpret the vanishing condition
in terms of injectivity of certain maps of linear series.
Afterwards we proceed to calculate the class of
the divisors in section \ref{sec:divisor-classes}.
The degeneracy calculation and the application to $\RgBar{8,3}$
can be found in section \ref{sec:application}.

\paragraph{Acknowledgements}

I would like to thank my Ph.D.\@ advisor Gavril Farkas
for introducing me to this topic and for many helpful discussions.
Thanks goes to Daniele Agostini who started the discussion which led
to the main application to $\RgBar{8,3}$.
I am also very grateful to the anonymous referee for pertinent comments
regarding the exposition, and for suggesting
Theorem \ref{thm:vector-bundle-vanishing} as well as a proof,
which improved and shortened section \ref{sec:divisors} considerably.
During the period of my Ph.D.\@ studies I was generously
supported by the IRTG 1800 of the Deutsche Forschungsgemeinschaft
and by the Berlin Mathematical School.

\section{Moduli spaces of quasi-stable level \texorpdfstring{$\bm{{\ell}}$}{l} curves}
\label{sec:moduli-spaces}

Let $\ell$ be a prime number.  By $\Rg{g,\ell}$ we denote the moduli space
of isomorphism classes of pairs $[C, \eta]$ where $C$ is a curve of genus $g$
and $\eta$ is a line bundle of order $\ell$ on $C$.
In this section we describe the basic facts about these moduli spaces
and their compactifications by quasi-stable level $\ell$ curves.

\subsection{Modular compactification}

Several constructions to compactify $\Rg{g,\ell}$ have been put forward.
Initially the theory was focused on the case $\ell = 2$ of Prym curves,
for which A.\@ Beauville put forward the theory of admissible covers
(\autocite{Beauville1977} and \autocite{ACV2003}).  It extends the modular description of
points in $\Rg{g,2}$ as étale double covers $\CoverCurve{C} \rightarrow C$ to stable curves.
Later, M.\@ Bernstein in her PhD thesis \autocite{Bernstein1999}
considered the normalization of $\MgBar{g}$ in the
function field of $\Rg{g,\ell}$.  Closed points of the ensuing compactification $\RgBar{g,\ell}$
correspond to stable curves with torsion line bundles on each component
and, over irreducible nodes, additionally the $\ell$-th roots of line bundles
of the form $\mathcal{O}_{\widetilde{C}}(ap + (\ell - a)q)$.  Here $p,q$
are the two points of the normalization lying over the node.

After the study of moduli spaces of roots of line bundles
by L.\@ Caporaso, C.\@ Casagrande and M.\@ Cornalba in \autocite{CCC2007},
it became clear what the right definition of limits of level $\ell \geq 3$ curves
should be.  The study of moduli spaces of these \emph{quasi-stable level $\ell$ curves}
was initiated in \autocite{CEFS2013}.
This very convenient modular interpretation
for the geometric points of $\RgBar{g,\ell}$ we are going to introduce here
and use subsequently.
\begin{definition}
  A \emph{quasi-stable} curve of genus $g$ 
  is a connected nodal curve $C$ of arithmetic genus $g$
  such that:
  \begin{enumerate}
  \item Every smooth rational component $E$ of $C$ meets the
    rest of $C$ in at least two points, i.e.,
    we have $n_E \coloneqq |E \cap \overline{(C \setminus E)}| \geq 2$.
  \item If $E$ and $E'$ are two such components
    with $n_E = n_{E'} = 2$,
    then we have $E = E'$ or $E \cap E' = \emptyset$.
  \end{enumerate}
  A smooth rational component $E$ with $n_E = 2$ is called \emph{exceptional}.
\end{definition}
Note that by blowing down all exceptional components of a quasi-stable curve we obtain a stable curve.
\begin{definition}
  A \emph{quasi-stable level $\ell$ curve} of genus $g$ is a triple $[C, \eta, \beta]$
  consisting of a quasi-stable curve $C$ of genus $g$,
  a line bundle $\eta\in\Pic^0(C)$
  and a sheaf homomorphism $\beta\colon \eta^{\otimes \ell} \rightarrow \mathcal{O}_C$,
  subject to the following conditions:
  \begin{enumerate}
  \item For each exceptional component $E$ of $C$ we have $\eta|_E = \mathcal{O}_E(1)$.
  \item For each non-exceptional component the morphism $\beta$ is an isomorphism.
  \item For each exceptional component $E$ and $\{p,q\} = E \cap \overline{C \setminus E}$ we have
    \[\ord_p(\beta) + \ord_q(\beta) = \ell\]
  \end{enumerate}
  A \emph{family of quasi-stable level $\ell$ curves} over a scheme $S$ is a triple
  $(\mathcal{C} \rightarrow S, \eta, \beta)$
  where $\mathcal{C} \rightarrow S$ is a a flat family of quasi-stable curves, $\eta$
  is a line bundle on $\mathcal{C}$ and $\beta\colon \eta^{\otimes \ell} \rightarrow \mathcal{O}_{\mathcal{C}}$
  is a sheaf homomorphism such that for each geometric fiber $C_s \rightarrow \{ s \} \subset S$
  the triple $(C_s, \eta|_{C_s}, \beta|_{C_s})$
  is a quasi-stable level $\ell$ curve.
\end{definition}
The fibered category of families of quasi-stable level $\ell$ curves defines
a Deligne--Mumford stack whose associated coarse moduli space we denote by $\Rootgl{g,\ell}$.
Since for $\ell > 3$ the singularities of $\Rootgl{g,\ell}$ are not normal,
the definition of the actual moduli space $\RgBar{g,\ell}$ is a bit more involved.
It arises as a connected component of the coarse moduli space $\MgBar{g}(\operatorname{B}\Z_\ell)$
of twisted level curves (\autocite{ACV2003}), which is a normalization of $\Rootgl{g,\ell}$.
In particular the treatment of the universal curve over the Deligne--Mumford stack $\RgBarStack{g,\ell}$
requires some further work.  We direct the reader to the extensive discussions in \autocite{Chiodo2008}
and \autocite{CEFS2013}.

\subsection{Boundary divisors}
\label{subsec:rgl-boundary}

Let $\RgForgetfulMap\colon \RgBar{g,\ell} \rightarrow \MgBar{g}$ be the forgetful map.
We study the boundary components of $\RgBar{g,\ell}$.
They lie over the boundary of $\MgBar{g}$, so we can examine
the components lying over $\Delta_i$ for $i = 0,\dotsc,\lfloor \frac{g}{2} \rfloor$.
Because of notational convenience sometimes boundary components of $\MgBar{g}$ and $\RgBar{g,\ell}$
will be denoted by the same symbols.  However it should always be clear from the context which space
we are considering.

\paragraph{The divisors $\bm{{\Delta_i}},\bm{{\Delta}_{g-i}},\bm{{\Delta}_{g:i}}$, $\bm{{i\geq 1}}$.}

First consider $i\geq 1$
and let $X \in \Delta_i$ be general, i.e., $X = C\cup D$ is the union of two curves of genera $i$ and $g-i$
meeting transversally in a single node.  The line bundle $\eta\in\Pic^0(X)$ on the corresponding level $\ell$ curve
is determined by its restrictions $\eta_C = \eta|_C$ and $\eta_D = \eta|_D$ satisfying
$\eta_C^{\otimes \ell} = \mathcal{O}_C$ and $\eta_D^{\otimes \ell} = \mathcal{O}_D$.

Either one of $\eta_C$ and $\eta_D$ (but not both) can be trivial, so $\RgForgetfulMap^\ast(\Delta_i)$ splits into three irreducible components $\Delta_i$, $\Delta_{g-i}$ and $\Delta_{i:g-i}$
where the general element in $\Delta_i$ is $[C\cup D, \eta_C \not= \mathcal{O}_C, \mathcal{O}_D]$,
the generic point of $\Delta_{g-i}$ is of the form
$[C\cup D, \mathcal{O}_C, \eta_D \not= \mathcal{O}_D]$
and the generic point of $\Delta_{i:g-i}$ looks like $[C\cup D, \eta_C \not= \mathcal{O}_C, \eta_D\not= \mathcal{O}_D]$.
Observe that for $i = 1$ and $\ell \geq 3$, due to the extra automorphism on elliptic tails, we have the pullback formula
$\RgForgetfulMap^\ast(\Delta_1) = 2 \Delta_1 + 2 \Delta_{1:g-1} + \Delta_{g-1}$ and the map $\RgForgetfulMap$
is ramified along $\Delta_1$ and $\Delta_{1:g-1}$.

\paragraph{The divisor $\bm{{\Delta_0''}}$.}

Now let $i = 0$.  The generic point of $\Delta_0$ in $\MgBar{g}$ is a one-nodal irreducible curve $C$
of geometric genus $g - 1$.  We first consider points of the form $[C, \eta]$ lying over $C$, i.e., without
an exceptional component.  Denote by $\nu\colon \normaliz{C} \rightarrow C$ the normalization
and by $p,q$ the preimages of the node.
Then we have an exact sequence
\begin{equation*}
  0 \rightarrow \CC^\ast \rightarrow \Pic^0(C) \xrightarrow{\nu^\ast}
  \Pic^0(\normaliz{C}) \rightarrow 0 
\end{equation*}
which restricts to
\begin{equation*}
  0 \rightarrow \Z/\ell\Z \rightarrow \Pic^0(C)[\ell] \xrightarrow{\nu^\ast} \Pic^0(\normaliz{C})[\ell] \rightarrow 0 
\end{equation*}
on the $\ell$-torsion part.  The group $\Z/\ell\Z$ represents the $\ell$ possible choices of gluing the fibers
at $p$ and $q$ for each line bundle in $\Pic^0(\normaliz{C})[\ell]$.
For the case $\nu^\ast \eta = \mathcal{O}_{\normaliz{C}}$ there are exactly
$\ell - 1$ possible choices of $\eta \not= \mathcal{O}_C$.  These curves $[C, \eta]$ correspond to
the order $\ell$ analogues of the classical \emph{Wirtinger double covers}
\begin{equation*}
  \normaliz{C}_1\amalg \normaliz{C}_2 / (p_1 \sim q_2, p_2 \sim q_1)  \xrightarrow{2:1}  \normaliz{C}/(p\sim q) = C
\end{equation*}
We denote by $\Delta_0''$ the closure of the locus of level $\ell$ Wirtinger covers.
Note that for $\ell > 3$ the divisor $\Delta_0''$ is not irreducible.  Indeed, up to
switching the role of the points $p$ and $q$ lying over the node,
the sections $s$ of an $\ell$-torsion line bundle
$\eta' \in \Pic^0(\widetilde{C})$ that descend to $C$ are determined by $s(p) = \xi^a s(q)$
where $\xi$ is an $\ell$-th root of unity and $1 \leq a \leq \ell - 1$.
Hence we get precisely $\lfloor \ell/2 \rfloor$ irreducible components and each
of them has order $2$ over $\Delta_0 \subset \MgBar{g}$.

\paragraph{The divisor $\bm{{\Delta_0'}}$.}

On the other hand, there are $\ell^{2(g-1)} - 1$ nontrivial elements in the group $\Pic^0(\normaliz{C})[\ell]$.
For each of them there are $\ell$ choices of gluing, so we have a total of $\ell \cdot (\ell^{2g - 2} - 1)$ choices for
$\eta \in \Pic^0(C)$ such that $\nu^\ast \eta \not= \mathcal{O}_{\normaliz{C}}$.
We let $\Delta_0'$ be the closure of the locus
of pairs $[C,\eta]$ such that $\nu^\ast\eta \not = \mathcal{O}_{\normaliz{C}}$.

\paragraph{The divisors $\bm{{\Delta_0^{(a)}}}$.}

We turn to the case of curves of the form $[X = \widetilde{C}\cup_{p,q} E, \eta]$ where $E$ is an exceptional component.
The stabilization of such a curve is again a one-nodal curve $C$.
Denote by $\beta$ the morphism $\eta^{\otimes \ell} \rightarrow \mathcal{O}_X$.
Since $\eta|_E = \mathcal{O}_E(1)$, we must have $\beta_{E\setminus\{p,q\}} = 0$
and $\deg(\eta^{\otimes \ell}|_{\normaliz{C}}) = -\ell$.
By swapping $p$ and $q$ if necessary, we can conclude that
$\eta^{\otimes \ell}|_{\normaliz{C}} = \mathcal{O}_{\normaliz{C}} ( - ap - (\ell - a)q )$
for some integer $a$ with $1 \leq a \leq \lfloor \ell/2\rfloor$.
There are $\ell^{2(g-1)}$ choices of square roots
of $\mathcal{O}_{\normaliz{C}}(-ap-(\ell-a)q)$ and each of these determines uniquely
a Prym curve $[X, \eta]$ of this form.
We denote the closure of the locus of such curves by $\Delta_0^{(a)}$.
Then the degree of $\Delta_0^{(a)}$ over $\Delta_0$ is $2 \ell^{2g - 2}$
for all $a$.  The factor $2$ arises because of the symmetry in $p$ and $q$.

\subsection{The canonical class}

Our goal is to show that $\RgBar{8,3}$ is of general type,
i.e., we have to show that the canonical class $K_{\RgBarDesing{8,3}}$
is big for some desingularization $\RgBarDesing{8,3}$ of $\RgBar{8,3}$.
An extension result for pluricanonical forms (Remark 3.5 in \autocite{CEFS2013})
shows that in fact we do not need to pass to a desingularization
but can perform all calculations on $\RgBar{8,3}$ directly.

Let us denote by $\delta_i$, $\delta_0'$ etc.\@ the rational divisor classes
associated to the respective boundary divisors.
The canonical class of $\RgBar{g,\ell}$ for $g \geq 4$ and $\ell \geq 3$
then has the following expression (see \autocite{CEFS2013}, Proposition 1.5):
\begin{equation*}
  K_{\RgBar{g,\ell}} = 13 \lambda - 2 (\delta_0' + \delta_0'')
  - (\ell + 1) \sum_{k = 1}^{\lfloor \ell/2 \rfloor } \delta_0^{(k)}
  - 2 \sum_{i = 1}^{\lfloor g/2 \rfloor} (\delta_i + \delta_{g-i} + \delta_{i:g-i})
  - \delta_{g-1}
\end{equation*}
Hence, if we can find an effective divisor
\begin{equation*}
  E \equiv a \lambda - b_0' \delta_0' - b_0'' \delta_0''
  - \sum_{k = 1}^{\lfloor \ell/2 \rfloor} b_0^{(k)} \delta_0^{(k)}
  - \sum_{i = 1}^{\lfloor g/2 \rfloor} (b_i \delta_i + b_{g-i} \delta_{g-i} + b_{i:g-i} \delta_{i:g-i})
\end{equation*}
such that
\begin{equation}
  \label{eq:mg-canonical-bounds}
  \frac{a}{b_0'},
  \frac{a}{b_0''},
  \frac{a}{b_i},
  \frac{a}{b_{g-i}},
  \frac{a}{b_{i:g-i}} < \frac{13}{2},\quad
  \frac{a}{b_{g-1}} < \frac{13}{3},\quad
  \frac{a}{b_0^{(k)}} < \frac{13}{\ell + 1}
\end{equation}
then it follows that we can write
\begin{equation*}
  K_{\RgBar{g,\ell}} \equiv \varepsilon \lambda + \alpha E + \beta D
\end{equation*}
where $D$ is supported on the boundary, $\alpha, \beta \geq 0$ and $\varepsilon > 0$.
Since the Hodge class $\lambda$ is big, $K_{\RgBar{g,\ell}}$ must be big as well.

Importantly, at least for $g \leq 23$, the only relevant data
are the coefficients of $\lambda$, $\delta_0'$, $\delta_0''$ and $\delta_0^{(k)}$:
\begin{lemma}[\autocite{CEFS2013}, Remark 3.5]
  \label{lem:irreducible-is-enough}
  Let $g \leq 23$ and $\ell \geq 2$.
  In order to prove that $K_{\RgBar{g,\ell}}$ is big
  it is enough to exhibit an effective divisor
  \begin{equation*}
    E \equiv a \lambda - b_0' \delta_0' - b_0'' \delta_0''
    - \sum_{k = 1}^{\lfloor \ell/2 \rfloor} b_0^{(k)} \delta_0^{(k)}
    - \sum_{i = 1}^{\lfloor g/2 \rfloor} (b_i \delta_i + b_{g-i} \delta_{g-i} + b_{i:g-i} \delta_{i:g-i})
  \end{equation*}
  with $a/b_0' < 13/2$, $a/b_0'' < 13/2$ and $a / b_0^{(k)} < 13/(\ell + 1)$
  for all $k = 1, \dotsc, \lfloor \ell/2 \rfloor$.
  The coefficients $b_i$, $b_{g-i}$ and $b_{i:g-i}$ are then automatically suitably bounded.
\end{lemma}

\section{Mukai bundles}
\label{sec:mukai-bundles}

Let $g = 6$ or $g = 8$.  We will always denote the Mukai bundle associated to
a curve $C$ by $E_C$.  By the results of \autocite{Mukai1993} it is possible to give
explicit Brill--Noether type conditions for a curve to arise as a complete intersection with
a Grassmannian:
\begin{theorem}[\autocite{Mukai1993}, Main Theorem]
  A curve $C$ of genus $8$ is a transversal linear section of the $8$-dimensional
  Grassmannian $G(2,6) \subseteq \P^{14}$ if and only if $C$ has no $\mathfrak{g}^2_7$.
\end{theorem}
\begin{theorem}[\autocite{Mukai1993}, §5]
  A curve $C$ of genus $6$ is the complete intersection of $G(2,5)$ and a
  $4$-dimensional quadric in $\P^9$ if and only if $W^1_4(C)$ is finite,
  i.e., $C$ is not trigonal, not a plane quintic and not bielliptic.
\end{theorem}
We get the vector bundles $E_C$ in question by restricting the tautological bundle
of the Grassmannian to $C$.  Importantly for us, it turns out that the existence of
a vector bundle with the right numerics is guaranteed by slightly weaker assumptions:
\begin{theorem}[\autocite{Mukai1993}, §5]
  \label{thm:mukai-bundle-g6}
  Let $C$ be a curve of genus $6$ which is neither trigonal
  nor a plane quintic. When $F$ runs over all stable rank $2$ bundles with canonical
  determinant on $C$, the maximum of $h^0(C, F)$ is equal to $5$.
  Moreover, such vector bundles $E_C$ on $C$ with $h^0(C, E_C) = 5$
  are unique up to isomorphism and generated by global sections.
\end{theorem}
\begin{theorem}[\autocite{Mukai1993}, §3]
  \label{thm:mukai-bundle-g8}
  Let $C$ be a curve of genus $8$ without a $\mathfrak{g}^1_4$.
  When $F$ runs over all semistable rank $2$ bundles with canonical determinant on $C$,
  the maximum of $h^0(C, F)$ is equal to $6$.  Moreover, such vector bundles
  $E_C$ on $C$ with $h^0(C, E_C) = 6$ are unique up to isomorphism
  and generated by global sections.
\end{theorem}
We denote the locus of curves satisfying the assumptions of Theorem \ref{thm:mukai-bundle-g6}
or \ref{thm:mukai-bundle-g8} by $\MukaiMg{g}$ and we set
$\MukaiRg{g,\ell} = \MukaiMg{g} \times_{\Mg{g}} \Rg{g,\ell}$.
The codimension of the complement of this locus is two:  In genus $6$ the trigonal locus
has codimension $2$ and the locus of plane quintics has codimension $3$.
In genus $8$, the tetragonal locus is also of codimension $2$.

We remark in passing that on genus $7$ and genus $9$ curves there also exist special
Mukai bundles.  They have rank $5$ and $3$, respectively.  Analogously to their counterparts
in genus $6$ and $8$ their global sections give embeddings of the curve, albeit in 
an orthogonal or a symplectic Grassmannian.
These bundles too exhibit interesting properties:  for instance
on a general genus $9$ curve they were used to give early counterexamples
to Mercat's conjecture (see \autocite{LMN2012}).

There is a more explicit construction of the bundles $E_C$ in question:
\begin{lemma}
  \label{lem:bundle-construction}
  Let $C$ be a curve of genus $6$, not trigonal and not a plane quintic,
  and $A \in W^1_4(C)$.  Set $L = \SerreDual{A}$.
  The bundle $E_C$ is given as the unique nontrivial extension of
  $L$ by $A$ with a $5$-dimensional space of global sections.
\end{lemma}
We quickly describe the main steps of the proof because they
will be useful in section \ref{sec:divisors}.  The full details can be found
in \autocite{Mukai1993}.  Consider any extension
\begin{equation*}
  0 \rightarrow A \rightarrow F \rightarrow L \rightarrow 0
\end{equation*}
and the resulting exact sequence in cohomology:
\begin{equation*}
  0 \rightarrow H^0(C, A) \rightarrow H^0(C, F) \rightarrow H^0(C, L)
  \xrightarrow{\delta_F} H^1(C, A) \rightarrow \dots
\end{equation*}
Then $h^0(C, F) \leq h^0(C, A) + h^0(C, L) = 5$ with equality if and only if $\delta_F = 0$.
By Serre duality we have
\begin{equation*}
  \Ext^1(L, A) \cong H^1(C, A \otimes L^{-1}) \cong H^0(C, L^{\otimes 2})^\vee
\end{equation*}
while the boundary homomorphism $\delta_F$ lies in
\begin{equation*}
  \Hom(H^0(C, L), H^1(C, A)) = \Hom(H^0(C,L), H^0(C, L)^\vee) = H^0(C, L)^\vee \otimes H^0(C, L)^\vee
\end{equation*}
We have a map
\begin{equation}
  \label{eq:extension-map}
  \Ext^1(L, A) = H^0(C, L^{\otimes 2})^\vee
  \rightarrow H^0(C, L)^\vee \otimes H^0(C, L)^\vee, \quad
\end{equation}
given by $F \mapsto \delta_F$, which is dual to the multiplication map of sections
\begin{equation*}
  H^0(C, L) \otimes H^0(C, L) \rightarrow H^0(C, L^{\otimes 2})
\end{equation*}
It turns out that the cokernel of this map is $H^0(C, L^{\otimes 2}) / \Sym^2 H^0(C, L)$,
which is $1$-dimen\-sional.  Hence, up to scaling there is a unique nonzero element in the kernel
of the map \eqref{eq:extension-map}.  It can be checked that this construction
does not depend on the choice of $A$.

In a completely analogous fashion, we get the following result in genus $8$:
\begin{lemma}
  Let $C$ be a curve of genus $8$ with $W^1_4(C) = \emptyset$
  and $A \in W^1_5(C)$.  Set $L = \SerreDual{A}$.
  The bundle $E_C$ is given as the unique nontrivial extension of
  $L$ by $A$ with a $6$-dimensional space of global sections.
\end{lemma}

\section{Constructing the divisors}
\label{sec:divisors}

In both genera, the slope of $E_C$ is
\begin{equation*}
  \mu(E_C) = \frac{\deg \det E_C}{\rk{E_C}} = \frac{2g - 2}{2} = g - 1
\end{equation*}
hence $\chi(E_C) = 0$.  We can therefore consider the virtual theta divisor of $E_C$
\begin{equation*}
  \Theta_{E_C} = \{ \xi \in \Pic^0(C) ~|~ H^0(C, E_C \otimes \xi) \not= 0\}
\end{equation*}
in $\Pic^0(C)$.  Since $E_C$ is a semistable rank $2$ vector bundle,
$\Theta_{E_C}$ is indeed of codimension one (see \autocite{Raynaud1982},
Proposition 1.6.2).  By intersecting $\Rg{g,\ell}$ and the locus
$\{ [C, \xi] ~|~ \xi \in \Theta_{E_C} \}$ in the universal Jacobian,
we obtain that
\begin{equation*}
  \MukaiDiv{g}{\ell} = \left\{ [C, \eta] \in \Rg{g,\ell}
  ~\left|~ H^0(C, E_C \otimes \eta) \not= 0 \right.\right\}
\end{equation*}
is of codimension at most one in $\Rg{g,\ell}$ and expected to be a divisor.

\subsection{Transversality}

We will show that on the general pair $[C, \eta] \in \Rg{g,\ell}$
we have $H^0(C, E_C \otimes \eta) = 0$.
More precisely, we will show that on the general smooth curve $C$
there exists an $\ell$-torsion bundle $\eta$ such that $H^0(C, E_C \otimes \eta) = 0$.

Let $C$ be a smooth and connected curve of genus $g$
and let $E$ be a semistable vector bundle of rank $r \geq 2$
and determinant $\det(E) = \vartheta^{\otimes r}$, where $\vartheta$
is a theta characteristic.  The \emph{theta divisor} of $\vartheta$ is
\begin{equation*}
  \Theta = \{ \xi \in \Pic^0(C) ~|~ H^0(C, \vartheta \otimes \xi) \not= 0 \}
\end{equation*}
Assume further that $h^0(C, E) \geq 1$ and that $E$ admits a theta-divisor,
i.e., the set
\begin{equation*}
  \Theta_E^{\mathrm{set}} = \{ \xi \in \Pic^0(C) ~|~ H^0(C, E \otimes \xi) \not= 0 \}
\end{equation*}
is a proper subset of $\Pic^0(C)$.
Then $\Theta_E^{\mathrm{set}}$ is the support of a natural
divisor $\Theta_E \in |r \Theta|$, called a \emph{theta divisor} of $E$
(see for instance \autocite{Beauville2004}).
We now consider the vector bundle $F \coloneqq E \oplus \vartheta^{\oplus(\ell - r)}$.
As a direct sum of semistable vector bundles of the same slope it is itself
semistable.  Furthermore, it admits an associated theta divisor $\Theta_F$
by assumption on $E$.
Then
\begin{equation*}
  \Supp \Theta_F = \Supp \Theta_E \cup \Supp \Theta
\end{equation*}
and we have
\begin{equation*}
  \Theta_F = \Theta_E + (\ell - r) \Theta \in |\ell \Theta|
\end{equation*}
We now let $J = \Pic^0(C)$ be the Jacobian of $C$
and $f_\ell\colon J \rightarrow \PP H^0(J, \mathcal{O}_J(\ell \Theta))^\ast = \PP^{\ell^g - 1}$ 
be the morphism defined by the linear system $|\ell \Theta|$.
Recall from \autocite{Mumford1966} that the
representation of the theta group of
$\mathcal{O}_J(\ell \Theta)$ is irreducible.
Any linear subspace of $\PP^{\ell^g - 1}$ containing the image
$f_\ell( J[\ell] )$ of the set of $\ell$-torsion points
of $J$ corresponds to an invariant subspace of the representation
of the theta group, hence $f_\ell( J[\ell] )$ is not contained in
a hyperplane.

Summarizing this discussion, we have:
\begin{theorem}
  \label{thm:vector-bundle-vanishing}
  Let $C$ be a smooth connected curve.
  Let $E$ be a semistable vector bundle on $C$ of rank $r \geq 2$ and
  determinant $\det(E) = \vartheta^{\otimes r}$ with $\vartheta$
  a theta characteristic.  Assume $H^0(C, E) \not= 0$ and
  $E$ admits a theta divisor.
  Then there exists a non-trivial $\eta \in \Pic^0(C)[\ell]$
  with $H^0(C, E \otimes \eta) = 0$.
\end{theorem}
\begin{proof}
  Consider as before $F = E \oplus L^{\oplus(\ell - r)}$.
  Then, by the above analysis, there is some torsion
  point $\eta \in \Pic^0(C)[\ell]$ which is not contained in
  the hyperplane of $\PP^{\ell^g - 1}$ defined by
  $\Theta_F = \Theta_E + (\ell - r)\Theta$.
  Hence $\eta$ is not contained in $\Theta_E$, which means
  $H^0(C, E \otimes \eta) = 0$.  By assumption
  $H^0(C, E \otimes \mathcal{O}_C) \not = 0$
  and therefore $\eta \not= \mathcal{O}_C$.
\end{proof}
In particular, we can apply the above theorem to $E_C$,
which is stable of rank $2$ and has canonical determinant,
hence satisfies all the assumptions.
We conclude that $\MukaiDiv{g}{\ell}$ is a divisor for every $\ell$
and both genera $g = 6$ and $g = 8$.

\subsection{Reinterpreting the divisor}

To calculate the divisor classes of $\MukaiDiv{g}{\ell}$,
it will be necessary first to give other characterizations of the pairs
$[C,\eta] \in \MukaiDiv{g}{\ell}$.

In the following discussion we will only consider $[C,\eta] \in \MukaiRg{g,\ell}$.
As discussed in Lemma \ref{lem:bundle-construction}, the bundle $E_C$ is an extension
\[ 0 \rightarrow A \rightarrow E_C \rightarrow \SerreDual{A} \rightarrow 0 \]
where $A\in W^1_4(C)$ if $g = 6$ and $A \in W^1_5(C)$ if $g = 8$.
After tensoring with $\eta$, the associated long exact sequence in cohomology starts with
\begin{equation*}
  0 \rightarrow H^0( C, A \otimes \eta ) \rightarrow H^0( C, E_C \otimes \eta )
  \rightarrow H^0( C, \SerreDual{A} \otimes \eta )
  \xrightarrow{\delta_{E_C\otimes \eta}} H^1( C, A \otimes \eta )
\end{equation*}
We immediately get:
\begin{lemma}
  \label{lem:divisor-boundarymap-condition}
  $[C,\eta] \in \MukaiDiv{g}{\ell}$ if and only if
  there exists an $A$ such that $H^0(C, A \otimes \eta) \not= 0$ or the boundary map
  \begin{equation}
    \label{eq:boundary-morphism}
    \delta_{E_C \otimes \eta}\colon H^0(C, \SerreDual{A} \otimes \eta)
    \rightarrow H^1(C, A \otimes \eta)
  \end{equation}
  is not an isomorphism.
\end{lemma}
Since $H^0(C, A \otimes \eta) \not= 0$ happens only on curves in a subvariety of
codimension at least $2$ in $\Rg{g,\ell}$, in what follows we will ignore the locus of such curves.
This does not affect divisor class calculations.
\begin{dummyversion}
  The locus
  \[ \left\{ (C, \eta) \in \MukaiRg{6,\ell} ~|~
    \exists A \in W^1_4(C)\colon H^0(A\otimes \eta)\not=0 \right\}\]
  has codimension $2$ in $\MukaiRg{6,\ell}$.
  A naive argument:  If $h^0( A \otimes \eta ) \geq 1$ for $\eta\in\Pic^0(C)$ then
  $\eta \in W^0_4 - W^1_4$, so it lives in the $4$-dimensional subspace $W^0_4 \subseteq \Pic^0(C)$,
  so the locus of curves having such an $\eta$ is of codimension $2$ in the universal
  Picard variety over $\mathcal{M}_6^0$.  Intersecting this with $\MukaiRg_{6,\ell}$
  gives a codimension $2$ locus there.
  A similar argument shows that in genus $8$, for $\mathfrak{g}^3_9$, we expect the
  codimension to be $3$.
\end{dummyversion}

In genus $6$, we can give another interpretation of $\MukaiDiv{6}{\ell}$.
Let $A \in W^1_4(C)$ and $L = \SerreDual{A}$.  By Riemann--Roch,
we have $h^0(C, L\otimes \eta) = 1$ and also $h^1(C, A\otimes \eta) = 1$.  So for \eqref{eq:boundary-morphism}
to be an isomorphism it is enough for it to be nonzero.
\begin{lemma}
  \label{lem:g6-divisor-condition}
  In the case $g = 6$, the boundary map
  $\delta_{E_C}\colon H^0(C, L\otimes \eta) \rightarrow H^1(C, A \otimes \eta)$
  is nonzero if and only if the multiplication map followed by projection
  \begin{equation}
    \label{eq:torsion-multiplication-map}
    H^0(C, L\otimes \eta) \otimes H^0(C, L\otimes \eta^{-1}) \xrightarrow{m_\eta}
    H^0(C, L^{\otimes 2}) \xrightarrow{p} H^0(C, L^{\otimes 2}) / \Sym^2 H^0(C, L)
  \end{equation}
  is an isomorphism.
\end{lemma}
\begin{proof}
  Since $C$ is not a plane quintic, $L$ is base point free, so it induces
  a morphism to $\P^2$.  The image is birational to $C$ if and only if
  $C$ is not trigonal and not bielliptic, so for a general genus $6$ curve
  $L$ induces a birational map to a $4$-nodal plane sextic.  This implies
  that the multiplication map $\Sym^2 H^0(C, L) \rightarrow H^0(C, L^{\otimes 2})$
  is injective.  So both domain and codomain of the map $p \circ m_\eta$ are $1$-dimensional.
  For a bielliptic curve $C \rightarrow E$
  the same conclusion holds by $H^0(C, L) \cong H^0(E, \mathcal{O}_E(1))$.
  Hence \eqref{eq:torsion-multiplication-map} is an isomorphism if and only if it is nonzero.
  
  Extensions of $L$ by $A$ and of $L \otimes \eta$ by $A \otimes \eta$ are both parametrized by
  \[\Ext^1(L, A) \cong \Ext^1(L \otimes \eta, A \otimes \eta)
  \cong H^1(C, A \otimes L^{-1}) \cong H^0(C, L^{\otimes 2})^\vee\]
  while the boundary morphism $\delta_{E_C\otimes \eta}$ lives in
  \begin{align*}
    \Hom( H^0(C, L\otimes \eta), H^1(C, A\otimes \eta) )
    & \cong H^0(C, L \otimes \eta)^\vee \otimes H^1(C, A \otimes \eta)\\
    & \cong H^0(C, L \otimes \eta)^\vee \otimes H^0(C, L \otimes \eta^{-1})^\vee
  \end{align*}
  and we have a map
  \begin{equation}
    \label{eq:ext-to-boundary-hom}
    \alpha\colon H^0(C, L^{\otimes 2})^\vee \rightarrow
    H^0(C, L \otimes \eta)^\vee \otimes H^0(C, L \otimes \eta^{-1})^\vee
  \end{equation}
  sending an extension $E\otimes \eta$ to the boundary
  homomorphism $\delta_{E\otimes \eta}$.  Note that $\alpha$ is the dual of the
  multiplication map $m_\eta$.  We denote by $[\alpha]$ the composition of $\alpha$
  with the dual of the projection $p$.

  The space $H^0(C, L^{\otimes 2})^\vee / \Sym^2 H^0(C, L)^\vee$ is generated
  by the class $[\phi_{E_C}]$ of the map corresponding to the Mukai bundle $E_C$
  (see the discussion after Lemma \ref{lem:bundle-construction}).
  Now \eqref{eq:torsion-multiplication-map} is the zero map if and only if
  the dual map $[\alpha]$ is the zero map if and only if $[\phi_{E_C}]$ is mapped
  to $0$ by $[\alpha]$, i.e., if $[\phi_{E_C}] \circ (p \circ m_\eta) = 0$.
  But this is exactly the boundary map $\delta_{E_C\otimes\eta}$
  given by the image of the extension $E_C\otimes \eta$ under \eqref{eq:ext-to-boundary-hom}.
\end{proof}
\begin{remark}
  For the case $\ell = 2$ further descriptions of the divisor exist.
  A general curve $[C,\eta] \in \MukaiDiv{6}{2}$ equivalently satisfies the following conditions:
  \begin{enumerate}[label=\alph*)]
  \item $C$ has a $4$-nodal plane sextic model with a totally tangent conic,
    i.e., there exists an $L\in W^2_6(C)$ inducing a birational map
    to $\Gamma \subseteq \P^2$, and a conic $Q \subseteq \P^2$ with
    $Q \cap \Gamma = 2D$ for some $D \in C^{(6)}$.
    This identification follows from Lemma \ref{lem:g6-divisor-condition}.
  \item The Prym map $\Rg{6,2} \rightarrow \Ag{5}$ is ramified at $[C,\eta]$
    (see \autocite{FGSV2014}, Theorem 8.1).
  \item $[C,\eta]$ is in the Prym--Brill--Noether divisor in $\Rg{6,2}$, i.e.
    \begin{equation*}
      \emptyset \not= V_3(C,\eta)
      = \left\{ L \in \Nm_f^{-1}(K_C) ~\left|~ h^0(\normaliz{C}, L) \geq r + 1,
      h^0(\normaliz{C}, L) \equiv r+1 \pmod{2} \right.\right\}
    \end{equation*}
    where $f\colon \normaliz{C} \rightarrow C$ is the étale double cover associated to $\eta$
    (\autocite{FGSV2014}, Theorem 0.4).
  \item $[C,\eta]$ is a section of a Nikulin surface (see \autocite{FV2012}, Theorem 0.5).
  \end{enumerate}
\end{remark}
\begin{remark}
  We can use the characterization of Lemma \ref{lem:g6-divisor-condition}
  to give another illustrative demonstration of the divisoriality of $\MukaiDiv{6}{\ell}$.
  This is achieved by explicitly constructing a pair $[C,\eta]$
  and a line bundle $L \in W^2_6(C)$ such that the map \eqref{eq:torsion-multiplication-map}
  is an isomorphism.  For brevity we skip the necessary proof that various moduli
  spaces of linear series and torsion points are irreducible.

  The construction of $[C, \eta, L]$ is as follows.
  Let $C$ be a plane quintic and choose any $L \in W^2_6(C)$.
  Let $\vartheta = \mathcal{O}_C(1)$ be the unique $\mathfrak{g}^2_5$ on $C$
  and recall that it is an odd theta characteristic.  Now $L$ can be
  written as $L = \vartheta \otimes \mathcal{O}_C(x)$
  for some point $x\in C$.  In particular, $x$ is a base point of $L$.
  Now choose an $\ell$-torsion bundle $\eta$
  on $C$ such that $h^0( C, \vartheta \otimes \eta ) = 0$.
  Then, by Riemann--Roch and Serre duality,
  $h^0( C, \vartheta \otimes \eta^{-1} ) = 0$ as well.
  This implies
  \[h^0( C, L \otimes \eta ) = h^0( C, L \otimes \eta^{-1} ) = 1\]
  and $x$ is neither a base point of $L \otimes \eta$ nor of $L \otimes \eta^{-1}$.
  Let $H^0( C, L \otimes \eta ) = \langle\sigma\rangle$ and
  $H^0( C, L \otimes \eta^{-1} ) = \langle\tau\rangle$
  and consider the map
  \[ \langle\sigma\rangle \otimes \langle\tau\rangle \rightarrow
    H^0( C, L^{\otimes 2} ) / \Sym^2 H^0( C, L )\]
  Observe that the multiplication map $\Sym^2 H^0( C, L) \rightarrow H^0( C, L^{\otimes 2} )$
  is injective, since
  \[ \Sym^2 H^0( C, \mathcal{O}_C(1) ) \rightarrow H^0( C, \mathcal{O}_C(2) )\]
  is an isomorphism ($C \subseteq \P^2$ is not contained in a quadric).
  The base locus of the image of $\Sym^2 H^0( C, L )$ in
  $H^0( C, L^{\otimes 2} ) = H^0( C, \omega_C(2x) )$ contains $2x$.
  But $\sigma \otimes \tau$, considered as a section in $H^0( C, L^{\otimes 2} )$,
  does not vanish at $x$.  Therefore it cannot be contained in the image of $\Sym^2( C, L )$, whence
  \[H^0( C, L^{\otimes 2} ) \cong \langle \sigma\otimes \tau\rangle \oplus \Sym^2 H^0( C, L )\]
  and we are done.
\end{remark}

\section{Divisor classes}
\label{sec:divisor-classes}

\subsection{Strategy}

An effective method to calculate divisor classes is to give a determinantal
description of the divisors, i.e., express them as the locus where a certain
morphism between vector bundles drops rank.  If the divisor involves
global sections of line bundles on curves, the vector bundles are usually
constructed over some space $\UnivGrd{r}{d}$ of linear series over the moduli space of curves.

To calculate the classes of $\MukaiDiv{g}{\ell}$ or some compactification of it,
a direct approach would be to try to use Lemma
\ref{lem:divisor-boundarymap-condition} and globalize the map
\begin{equation*}
  \delta_{E_C \otimes \eta}\colon H^0(C, \SerreDual{A} \otimes \eta)
  \rightarrow H^1(C, A \otimes \eta)
\end{equation*}
to a morphism of vector bundles over
$\UnivGrdTors{r}{d}{\ell} = \UnivGrd{r}{d}\times_{\Mg{g}} \Rg{g,\ell}$.
A naive attempt is to pass to the moduli stacks and to try to create a global extension
\begin{equation*}
  0 \rightarrow \mathcal{A} \rightarrow \mathcal{E} \rightarrow
  \omega_\UnivCrdTorsMorph \otimes \mathcal{A}^{-1} \rightarrow 0,
\end{equation*}
on the universal curve
$\UnivCrdTorsMorph\colon \UnivCrdTors{r}{d}{\ell} \rightarrow \UnivGrdTors{r}{d}{\ell}$,
tensor it by the universal $\ell$-torsion bundle $\PrymPoincare$ and use the map
induced by the long exact sequence of the pushforward $\UnivGrdTorsMorph_\ast$
where $\UnivGrdTorsMorph\colon \UnivGrdTors{r}{d}{\ell} \rightarrow \Rg{g,\ell}$.

However, this naive approach must fail.  The bundle $E_C$, as an extension of $\SerreDual{A}$ by $A$,
is only defined up to isomorphism on each curve and the choice can not
be made globally on the whole moduli space.
It is possible, though, to give a choice-free description of the condition that
the boundary morphism induced by $E_C \otimes \eta$ is not an isomorphism.

To this end, let $L = \SerreDual{A}$ and observe that the codomain of $\delta_{E_C \otimes \eta}$ is
\begin{equation*}
  H^1(C, A \otimes \eta) \cong H^0(C, L \otimes \eta^{-1})^\vee
\end{equation*}
by Serre duality.  Now the map
\begin{equation}
  \label{eq:degen-map-on-fibers}
  H^0(C, L \otimes \eta) \otimes
  \left( \frac{ H^0(C, L^{\otimes 2}) }{ \Sym^2 H^0(C, L) } \right)^\vee
  \rightarrow H^0(C, L \otimes \eta^{-1})^\vee
\end{equation}
can be defined canonically by setting
\begin{equation*}
  s \otimes f \mapsto [ t \mapsto f( s \cdot t ) ]
\end{equation*}
The quotient that appears in \eqref{eq:degen-map-on-fibers}
can be seen as encoding the $\CC^\ast$ of possible choices
for $E_C$ in $\Ext^1(L, A)$.
It is clear that the map \eqref{eq:degen-map-on-fibers} is an isomorphism if
and only if $\delta_{E_C \otimes \eta}$ is.
Since there are no choices involved in defining the map, we can readily globalize it.

\subsection{Definition of the degeneracy locus}

We first construct an appropriate partial compactification of $\Rg{g,\ell}$
where the class calculations can be carried out.
Here we use a setup similar to \autocite{FL2010} and \autocite{CEFS2013}.
Again denote by $\RgForgetfulMap\colon \Rg{g,\ell} \rightarrow \Mg{g}$ the forgetful map.
Let $\RgPrime{6,\ell} = \RgPCZero{6,\ell} \cup \RgForgetfulMap^\ast ( \BoundaryPCZero )$,
where $\RgPCZero{6,\ell}$ is the locus of smooth curves $[C,\eta]$ such that
$\dim W^2_6(C) = 0$ and $H^1(C, L \otimes \eta) = 0$
for all $L \in W^2_6(C)$, and $\BoundaryPCZero$ is the locus of irreducible one-nodal
curves $[C_{pq}] \in \Delta_0$ where $[C,p,q] \in \Mg{5,2}$ is Petri general.

Similarly, let $\RgPrime{8,\ell} = \RgPCZero{8,\ell} \cup \RgForgetfulMap^\ast( \BoundaryPCZero )$
be the locus of smooth curves $[C,\eta] \in \Rg{8,\ell}$ such that $\dim W^3_9(C) = 0$,
and $H^1(C, L \otimes \eta) = 0$ for all $L \in W^3_9(C)$, while $\BoundaryPCZero$
is the locus of curves $[C_{pq}]$ with $[C,p,q] \in \Mg{7,2}$ Petri general.
Observe that in both cases the complement of $\RgPrime{g,\ell}$ in
$\Rg{g,\ell}\cup \RgForgetfulMap^\ast( \Delta_0 )$ has codimension $2$,
so divisor class calculations will not be affected
(use Mumford's theorem, Theorem 4.1 in \autocite{Mukai1993} and the discussion
in section 8 of \autocite{FGSV2014}).

We are now in a position to provide a determinantal description of the divisor
$\MukaiDiv{g}{\ell}$.  To this end, we will construct a morphism of vector
bundles of the same rank over $\RgPrime{g,\ell}$ such that on fibers it corresponds
exactly to the map in \eqref{eq:degen-map-on-fibers}.
Then $\MukaiDivBar{g}{\ell}$ will be contained in the first degeneracy locus of this morphism
and its class can be calculated using Porteous' formula.

The setup is almost the same for both genera.  In genus $6$ we let $r = 2$, $d = 6$
and in genus $8$ we let $r = 3$, $d = 9$.
Now let $\UnivGrdTors{r}{d}{\ell}$ be the moduli stack of triples $[C, \eta, L]$
over $\RgPrime{g,\ell}$ where $L \in W^r_d(C)$ and
let $\UnivGrdTorsMorph\colon \UnivGrdTors{r}{d}{\ell} \rightarrow \RgPrime{g,\ell}$
be the morphism forgetting the $\mathfrak{g}^r_d$.
Denote further by
$\UnivCrdTorsMorph\colon \UnivCrdTors{r}{d}{\ell} \rightarrow \UnivGrdTors{r}{d}{\ell}$
the universal curve and let $\PoincareBundle$ be the universal $\mathfrak{g}^r_d$.
We also have the universal $\ell$-torsion bundle $\PrymPoincare$ over $\UnivCrdTors{r}{d}{\ell}$.
We will slightly abuse notation and denote the pullbacks of $\lambda$, $\delta_0'$, $\delta_0''$
and $\delta_0^{(a)}$ by $\UnivGrdTorsMorph$ by the same symbols, respectively.

By Grauert's theorem, $\UnivCrdTorsMorph_\ast( \PoincareBundle^{\otimes i} )$ is a
vector bundle for $i = 1,2$.  However, we do not know this for
$\UnivCrdTorsMorph_\ast( \PoincareBundle \otimes \PrymPoincare )$, since the
dimension of $H^0(C, L \otimes \eta)$ jumps on fibers over the whole boundary divisor $\Delta_0''$:
\begin{lemma}
  Let $g = 6$ and $[C,\eta] \in \Delta_0''$.  Then for any $L \in W^2_6(C)$
  we have $h^0(C, L \otimes \eta) = 2$.  Likewise, for $g = 8$ and any $L \in W^3_9(C)$
  on $[C,\eta] \in \Delta_0''$ we have $h^0(C, L \otimes \eta) = 3$.
\end{lemma}
\begin{proof}
  Let $\nu\colon \normaliz{C} \rightarrow C$ be the normalization of $C$ and $x$ be the node.
  Then $\nu^\ast \eta = \mathcal{O}_{\normaliz{C}}$ and $\nu^\ast L \in W^r_{d}(\normaliz{C})$,
  since $\normaliz{C}$ is Brill--Noether general.
  From the exact sequence
  \begin{equation*}
    0 \rightarrow \mathcal{O}_C \rightarrow \nu_\ast \mathcal{O}_{\widetilde{C}}
    \xrightarrow{e} \CC_x \rightarrow 0
  \end{equation*}
  we get
  \begin{equation*}
    0 \rightarrow L\otimes \eta \rightarrow \nu_\ast \nu^\ast L \xrightarrow{e'}
    L\otimes \eta|_x \rightarrow 0
  \end{equation*}
  and taking long exact sequence in cohomology we obtain
  \begin{equation*}
    0 \rightarrow H^0(C, L\otimes \eta) \rightarrow H^0(\normaliz{C}, \nu^\ast L)
    \xrightarrow{H^0(e')} \CC
  \end{equation*}
  Now $H^0(e)$ is the zero map, hence $H^0(e')$ must be nonzero and we get
  \begin{equation*}
    h^0(C, L\otimes \eta) = h^0(\normaliz{C}, \nu^\ast L) - 1 = r + 1 - 1 = r\qedhere
  \end{equation*}
\end{proof}
This shows that $R^1 \UnivCrdTorsMorph_\ast( \PoincareBundle \otimes \PrymPoincare )$ is
supported
on $\Delta_0''$ and of rank $1$ over there.
On the other hand, we do not know whether
$\UnivCrdTorsMorph_\ast( \PoincareBundle \otimes \PrymPoincare )$ is a vector bundle.
However, it is torsion-free, hence locally free in codimension $1$ and we can
throw out the at most codimension $2$ loci in $\RgPrime{g,\ell}$ where the rank jumps.
This will not affect our divisor class calculations.
Hence we will assume $\UnivCrdTorsMorph_\ast( \PoincareBundle \otimes \PrymPoincare )$
and $\UnivCrdTorsMorph_\ast( \PoincareBundle \otimes \PrymPoincare^{-1} )$
are vector bundles.
Now let
\begin{equation*}
  \mathcal{E} = \UnivCrdTorsMorph_\ast( \PoincareBundle \otimes \PrymPoincare ) \otimes
  \left( \UnivCrdTorsMorph_\ast( \PoincareBundle^{\otimes 2} ) /
    \Sym^2 \UnivCrdTorsMorph_\ast( \PoincareBundle ) \right)^\vee
\end{equation*}
and
\begin{equation*}
  \mathcal{F} = \left( \UnivCrdTorsMorph_\ast\big(
    \PoincareBundle \otimes \PrymPoincare^{-1} \big) \right)^\vee
\end{equation*}
We obtain a morphism 
\begin{equation}
  \label{eq:degenmorph}
  \DegenMorph_{g,\ell} \colon \mathcal{E} \rightarrow \mathcal{F}
\end{equation}
whose first degeneracy locus $Z_1(\DegenMorph_{g,\ell})$, pushed forward by $\UnivGrdTorsMorph$
and restricted to the Mukai locus $\MukaiRg{g,\ell}$,
coincides with our divisor $\MukaiDiv{g}{\ell}$.

\subsection{Calculation of the classes}

First we apply Porteous' formula to the morphism \eqref{eq:degenmorph} to obtain
\begin{equation*}
  [Z_1(\DegenMorph_{g,\ell})] = c_1(\mathcal{F} - \mathcal{E}) = c_1(\mathcal{F}) - c_1(\mathcal{E})
\end{equation*}
Using the elementary fact
\begin{equation*}
  c_1(\Sym^2 \mathcal{G}) = (\rk(\mathcal{G}) + 1) c_1(\mathcal{G})
\end{equation*}
for a vector bundle $\mathcal{G}$ we can write
\begin{equation*}
  c_1(\mathcal{E}) =
  c_1( \UnivCrdTorsMorph_\ast( \PoincareBundle \otimes \PrymPoincare ) )
  - (r-1) c_1( \UnivCrdTorsMorph_\ast( \PoincareBundle^{\otimes 2} ) )
  + (r-1)(r+2) c_1( \UnivCrdTorsMorph_\ast( \PoincareBundle ) )
\end{equation*}
We use Grothendieck--Riemann--Roch to calculate the Chern classes in these expressions.
Let $\omega_\UnivCrdTorsMorph$ be the relative dualizing sheaf of $\UnivCrdTorsMorph$ and
consider the classes
\begin{equation*}
  \mathfrak{a} = \UnivCrdTorsMorph_\ast(c_1^2(\PoincareBundle)),\quad
  \mathfrak{b} = \UnivCrdTorsMorph_\ast(c_1(\PoincareBundle)\cdot c_1(\omega_\UnivCrdTorsMorph)),\quad
  \mathfrak{c} = c_1(\UnivCrdTorsMorph_\ast(\PoincareBundle))
\end{equation*}
in $A^1(\UnivGrdTors{r}{d}{\ell})$.  Furthermore, let
$\mathfrak{d} =
c_1 \left( R^1 \UnivCrdTorsMorph_\ast( \PoincareBundle \otimes \PrymPoincare )\right)$.
\begin{dummyversion}
  \begin{remark}
    We have $\mathfrak{d} = \alpha \delta_0''$ for some $\alpha \geq 0$,
    probably $\mathfrak{d} = \delta_0''$. 
  \end{remark}
\end{dummyversion}
For brevity, set
\begin{equation*}
  \rho = \sum_{a = 1}^{\lfloor \ell/2 \rfloor} \frac{a(\ell - a)}{\ell} \delta_0^{(a)}
\end{equation*}
Applying Grothendieck--Riemann--Roch and using \autocite{CEFS2013}, Proposition 1.6, we get
\begin{align*}
  c_1( \UnivCrdTorsMorph_\ast( \PoincareBundle \otimes \PrymPoincare^{\pm 1} ) )
  & = \lambda + \frac{1}{2} \mathfrak{a} - \frac{1}{2} \mathfrak{b}
  - \frac{1}{2} \rho + \mathfrak{d}\\
  c_1( \UnivCrdTorsMorph_\ast( \PoincareBundle^{\otimes 2} ) )
  & = \lambda + 2 \mathfrak{a} - \mathfrak{b}
\end{align*}
\begin{dummyversion}
  \begin{remark}
    First I thought that we don't get a map
    \begin{equation*}
      \UnivCrdTorsMorph_\ast\big( \PoincareBundle \otimes \PrymPoincare^{-1} \big)
      \otimes 
      \UnivCrdTorsMorph_\ast\big( \PoincareBundle \otimes \PrymPoincare \big)
      \rightarrow
      \UnivCrdTorsMorph_\ast ( \PoincareBundle^{\otimes 2} )
    \end{equation*}
    and that we need to use
    \begin{equation*}
      \UnivCrdTorsMorph_\ast\big( \PoincareBundle \otimes \PrymPoincare^{\otimes(\ell-1)} \big)
      \otimes 
      \UnivCrdTorsMorph_\ast\big( \PoincareBundle \otimes \PrymPoincare \big)
      \rightarrow
      \UnivCrdTorsMorph_\ast ( \PoincareBundle^{\otimes 2} )
    \end{equation*}
    because I thought the contributions of $\PrymPoincare$ in the divisor
    would cancel out and we'd get a contradiction to the slope conjecture/theorem.
    It turns out I simply miscalculated.  But it is worth noting still that
    $\PrymPoincare^{\otimes \ell} \rightarrow \mathcal{O}_X$ is \emph{not}
    an isomorphism on the universal curve $X$.  There is no
    map in the other direction.  But it should be injective at least.
  \end{remark}
\end{dummyversion}
Putting everything together, we obtain
\begin{equation}
  \label{eq:degenclass}
  [Z_1(\DegenMorph_{g,\ell})] = (r-3)\lambda + (2r-3)\mathfrak{a}
  - (r-2)\mathfrak{b} - (r^2 + r - 2)\mathfrak{c} - 2\mathfrak{d}
  + \rho
\end{equation}
\begin{lemma}
  \label{lem:class-expressions}
  For $g = 6$ we have
  \begin{equation*}
    \UnivGrdTorsMorph_\ast( \mathfrak{a} )
    = -93 \cdot \lambda + \frac{23}{2}\RgForgetfulMap^\ast(\delta_0),\quad
    \UnivGrdTorsMorph_\ast( \mathfrak{b} )
    = -\frac{3}{2} \lambda + \frac{3}{4}\RgForgetfulMap^\ast(\delta_0),\quad
    \UnivGrdTorsMorph_\ast( \mathfrak{c} )
    = -\frac{133}{4} \lambda + \frac{33}{8}\RgForgetfulMap^\ast(\delta_0)
  \end{equation*}
  and for $g = 8$ we have
  \begin{equation*}
    \UnivGrdTorsMorph_\ast( \mathfrak{a} )
    = -267 \cdot \lambda + \frac{69}{2} \RgForgetfulMap^\ast(\delta_0),\quad
    \UnivGrdTorsMorph_\ast( \mathfrak{b} )
    = 3 \cdot \lambda + \frac{3}{2} \RgForgetfulMap^\ast(\delta_0),\quad
    \UnivGrdTorsMorph_\ast( \mathfrak{c} )
    = -100 \cdot \lambda + 13 \RgForgetfulMap^\ast(\delta_0)
  \end{equation*}
\end{lemma}
\begin{proof}
  Use the machinery of \autocite{Farkas2009}, in particular Lemma 2.6, Lemma 2.13 and Proposition 2.12.
\end{proof}
\begin{remark}
  A different choice of Poincaré bundle $\PoincareBundle$ affects
  the classes $\mathfrak{a}$, $\mathfrak{b}$ and $\mathfrak{c}$.  However,
  the class of the degeneracy locus of $\DegenMorph_{g,\ell}$ is independent of this choice
  (see the discussion before Theorem 2.1 in \autocite{Farkas2009}).
\end{remark}
Now we only need to pushforward $[Z_1(\DegenMorph_{g,\ell})]$ by $\UnivGrdTorsMorph$ to $\RgPrime{g,\ell}$,
which has the effect of multiplying the coefficients of $\lambda$, $\delta_0''$ and $\delta_0^{(a)}$ in \eqref{eq:degenclass} by
the degree of $\UnivGrdTorsMorph$.  This degree is $5$ in the case of $g = 6$
and $14$ in the case of $g = 8$ (the respective number of $\mathfrak{g}^r_d$
on the general curve).  Plug in the expressions of Lemma
\ref{lem:class-expressions} to obtain:
\begin{theorem}
  The class of the degeneracy locus $\UnivGrdTorsMorph_\ast Z_1(\DegenMorph_{6,\ell})$ is 
  \begin{equation*}
    [\MukaiDivBar{6}{\ell}]^\virt = 35 \lambda - 5 (\delta_0' + 3\delta_0'')
    - \frac{5}{\ell} \sum_{a = 1}^{\lfloor \ell/2 \rfloor}( \ell^2 - a\ell + a^2 )\delta_0^{(a)}
  \end{equation*}
\end{theorem}
\begin{theorem}
  The class of the degeneracy locus $\UnivGrdTorsMorph_\ast Z_1(\DegenMorph_{8,\ell})$ is 
  \begin{equation*}
    [\MukaiDivBar{8}{\ell}]^\virt = 196 \lambda - 28 (\delta_0' + 2\delta_0'')
    - \frac{14}{\ell} \sum_{a = 1}^{\lfloor \ell/2 \rfloor} (2\ell^2 - a\ell + a^2) \delta_0^{(a)}
  \end{equation*}
\end{theorem}
In particular, since
$\UnivGrdTorsMorph_\ast Z_1(\DegenMorph_{g,\ell}) \cap \Rg{g,\ell} = \MukaiDiv{g}{\ell}$,
the class $[\MukaiDivBar{g}{\ell}]^\virt - n[\MukaiDivBar{g}{\ell}]$ is effective
and entirely supported on the boundary of $\RgPrime{g,\ell}$ for some $n \geq 1$.
\begin{remark}
  The morphism $\DegenMorph_{g,\ell}$ is degenerate over the boundary component $\Delta_0''$,
  with order $1$ for $g = 6$ and order $2$ for $g = 8$.
  We can therefore subtract an additional $5 \delta_0''$ and $28 \delta_0''$, respectively.
\end{remark}
\begin{remark}
  The coefficients appearing in the expression of $\MukaiDivBar{6}{\ell}$ are divisible
  by $5$, which is exactly the degree of the map $\UnivGrdTorsMorph\colon \UnivGrdTors{2}{6}{\ell} \rightarrow \RgPrime{6,\ell}$.
  This can be explained by observing that the boundary morphism \eqref{eq:boundary-morphism}
  fails to be an isomorphism for some $A \in W^1_4(C)$ if and only if $H^0(C, E_C \otimes \eta) \not= 0$.
  But since $E_C$ does not depend on the choice of $A$, the morphism surprisingly fails to be bijective
  for \emph{all} $A \in W^1_4(C)$.

  Similarly, the coefficients for $\MukaiDivBar{8}{\ell}$ are divisible by $28 = 2 \cdot 14$,
  where $14 = \deg(\UnivGrdTorsMorph)$.  Observe that by Serre duality, $\chi(E_C \otimes \eta) = 0$,
  and by the isomorphism $E_C^\vee \otimes \omega_C \cong E_C$, we have $H^0(C, E_C \otimes \eta) = 0$
  if and only if $H^0(C, E_C \otimes \eta^{-1}) = 0$.  This explains the additional factor of two.
\end{remark}
\begin{dummyversion}
  \begin{corollary}
    The pushforward of the class of $\MukaiDivBar{6}{2}$ to $\MgBar{6}$ has class
    \begin{equation*}
      143325\cdot \lambda - 17930 \cdot \delta_0
    \end{equation*}
    and slope barely less than $8$.  This shows that the trigonal locus
    is contained in this divisior.

    For $\ell = 3$ we have the pushforward class
    \begin{equation*}
      18600400 \cdot \lambda - 226357 \cdot \delta_0
    \end{equation*}
    which has slope $8.217\dots$.

    For $g = 8$ and $\ell = 2$ the pushforward has class
    \begin{equation*}
      12844860\cdot \lambda - 1720348 \cdot \delta_0
    \end{equation*}
    and for $\ell = 3$:
    \begin{equation*}
      8437157120 \cdot \lambda - 1116026184 \cdot \delta_0
    \end{equation*}
    So for $\ell = 2,3$ this contains the trigonal locus.
  \end{corollary}
\end{dummyversion}

\section{Application to the birational geometry of modular varieties}
\label{sec:application}

\subsection{An improvement of existing divisor classes}

Recall the following result:
\begin{theorem}[\autocite{CEFS2013}, Theorem 0.7]
  \label{thm:cefs-dgl}
  Set $g = 2i + 2 \geq 4$ and $\ell \geq 3$ such that
  $i \equiv 1 \mod 2$ or $\binom{2i-1}{i} \equiv 0 \mod 2$.
  The virtual class of the closure in $\RgPrime{g,\ell}$
  of the locus $\mathcal{D}_{g,\ell}$ of level $\ell$ curves
  $[C, \eta] \in \Rg{g,\ell}$ such that
  $K_{i,1}(C; \eta^{\otimes (\ell - 2)}, K_C \otimes \eta) \not= 0$ is equal to
  \begin{align*}
    [\overline{\mathcal{D}}_{g,\ell}]^\virt = \frac{1}{i-1}
    \binom{2i-2}{i}
    \bigg( &
      (6i + 1) \lambda - i(\delta_0' + \delta_0'') \\
      & - \frac{1}{\ell} \sum_{a = 1}^{\lfloor \frac{\ell}{2} \rfloor}
      (i\ell^2 + 5a^2i - 5ai\ell - 2a^2 + 2a\ell) \delta_0^{(a)}
    \bigg)
  \end{align*}
\end{theorem}
Here $K_{i,j}(C, \eta^{\otimes (\ell - 2)}, K_C \otimes \eta)$ is the Koszul cohomology
group defined in \autocite{Green1984}.
We quickly sketch how this result was obtained.
For a globally generated vector bundle $E$ on $C$, let $M_E$
be the kernel of the surjective evaluation map
\begin{equation*}
  H^0(C, E) \otimes \mathcal{O}_C \rightarrow E \rightarrow 0
\end{equation*}
Then by standard arguments, e.g.\@ \autocite{AN2010},
we have an identification
\begin{equation*}
  K_{i,1}\big(C; \eta^{\otimes (\ell - 2)}, K_C \otimes \eta\big)
  = H^0\big(C, \wedge^i M_{K_C \otimes \eta} \otimes K_C \otimes \eta^{-1}\big)
\end{equation*}
This in turn can be identified with the kernel of the map
\begin{equation*}
  \wedge^i H^0(C, K_C \otimes \eta) \otimes H^0(C, K_C \otimes \eta^{-1})
  \rightarrow H^0(C, \wedge^{i-1} M_{K_C \otimes \eta} \otimes K_C^{\otimes 2})
\end{equation*}
Note that the domain and the target are vector spaces of the same dimension.
This map is then globalized to a map $\chi$ between vector bundles
of the same rank over $\RgPrime{g,\ell}$ and its first degeneracy locus
can be calculated using Porteous' formula to obtain the class of Theorem \ref{thm:cefs-dgl}.

We will now show that $\chi$ is degenerate along the boundary divisors $\Delta_0^{(a)}$
by calculating a lower bound on the dimension of
$V \coloneqq H^0(X, \wedge^i M_{\omega_X \otimes \eta} \otimes \omega_X \otimes \eta^{-1})$
for a general curve $[X, \eta] \in \Delta_0^{(a)}$.  The result does
not depend on $a$.

Let $X = C \cup_{p,q} E$ where $E \cong \PP^1$ is exceptional.
Observe that $\omega_X|_E = \mathcal{O}_E$ while $\omega_X|_C = K_C(p+q)$.
One then calculates that
\begin{equation*}
  M_{\omega_X \otimes \eta}|_C = M_{K_C(p+q) \otimes \eta_C}
  \text{ and }
  M_{\omega_X \otimes \eta}|_E = \mathcal{O}_E(-1) \oplus \mathcal{O}_E^{\oplus (g - 3)}
\end{equation*}
We let $M \coloneqq M_{K_C(p+q) \otimes \eta_C}$.
By the Mayer--Vietoris sequence $V$ is the kernel of
\begin{align*}
  H^0\left( C, \wedge^i M \otimes K_C(p+q) \otimes \eta_C^{-1} \right)
  \oplus
  H^0\big( E, \wedge^i(\mathcal{O}_E(-1) \oplus \mathcal{O}_E^{\oplus(g-3)}) \otimes \mathcal{O}_E(-1) \big)\\
  \rightarrow H^0\big( \wedge^i M_{\omega_X \otimes \eta} \otimes \omega_X \otimes \eta^{-1} \big|_{p+q} \big)
\end{align*}
Since $\wedge^i M_{\omega_X \otimes \eta}$ has rank $\binom{2i}{i}$,
the latter space has dimension $2\cdot\binom{2i}{i}$,
while the bundle on $E$ has no sections.  Using Riemann--Roch, we calculate
\begin{align*}
  h^0\left( C, \wedge^i M \otimes K_C(p+q) \otimes \eta_C^{-1} \right)
  & \geq -(2g - 3)\binom{2i-1}{i-1} + \binom{2i}{i}(2g - 1)
    + \binom{2i}{i}(2 - g)\\
  & = 5 \binom{2i-1}{i-1}
\end{align*}
hence the kernel has dimension at least
\begin{equation*}
  5 \binom{2i-1}{i-1} - 2\binom{2i}{i} = \binom{2i-1}{i-1}
\end{equation*}
and therefore $\chi$ is degenerate to this order on the
boundary $\Delta_0^{\mathrm{ram}}$.  We have proved:
\begin{proposition}
  The divisor class
  $[\overline{\mathcal{D}}_{g,\ell}]^\virt
  - \binom{2i-1}{i-1} \sum_{a = 1}^{\lfloor \ell/2 \rfloor} \delta_0^{(a)}$
  is effective.
\end{proposition}
\begin{example}
  For $g = 8$ (i.e.\@ $i = 3$) and $\ell = 3$ we obtain that
  \begin{equation*}
    [\overline{\mathcal{D}}_{8,3}]^\virt - 10 \delta_0^{(1)} \equiv 38 \lambda - 6(\delta_0' + \delta_0'') - \frac{32}{3}\delta_0^{(1)}
  \end{equation*}
  is effective.
\end{example}

\subsection{Degeneracy of \texorpdfstring{$\bm{{\MukaiDivBar{8}{3}}}$}{D83} on the boundary}

Recall that our strategy to calculate the divisor class
of $\MukaiDivBar{8}{3}$ was to globalize the map
\begin{equation*}
  H^0(C, L \otimes \eta) \otimes
  \left( \frac{ H^0(C, L^{\otimes 2}) }{ \Sym^2 H^0(C, L) } \right)^\vee
  \rightarrow H^0(C, L \otimes \eta^{-1})^\vee
\end{equation*}
where $L \in W^3_9(C)$.
Using Tensor-$\Hom$ adjunction, this map fails to be injective
if and only if the bilinear map corresponding to the multiplication map
\begin{equation}
  \label{eq:multiplication-boundary}
  \mu_{[C, \eta, L]}\colon
  H^0(C, L \otimes \eta) \otimes H^0(C, L \otimes \eta^{-1})
  \rightarrow
  H^0(C, L^{\otimes 2}) / \Sym^2 H^0(C, L)
\end{equation}
is degenerate.  We will in fact show that for general $[X, \eta] \in \Delta_0^{(1)}$
and $L \in W^3_9(X)$ the map $\mu_{[X, \eta, L]}$ is the zero map,
i.e., the image of
$H^0(X, L \otimes \eta) \otimes H^0(X, L \otimes \eta^{-1}) \rightarrow H^0(X, L^{\otimes 2})$
is contained in the image of $\Sym^2 H^0(X, L)$.

A line bundle $L \in W^3_9(X)$ can be described as follows.  The restriction
$L_C = L|_C$ of $L$ to $C$ has the property $h^0(C, L_C(-p-q)) = 3$.
Since $L$ restricts to $\mathcal{O}_E$ on $E$,
any global section $s \in H^0(X, L)$ restricts to a constant on $E$ and hence the restriction
to $C$ has the same value at $p$ and $q$.
If $s|_E = 0$, then $s|_C \in H^0(C, L_C(-p-q))$.  If $s|_E$ is instead a nonzero constant
then $s|_C$ is a global section of $L_C$ which vanishes neither at $p$ nor at $q$.
Fix such a section $\sigma$.  Then we have an isomorphism
\begin{equation*}
  H^0(X, L) \cong H^0(C, L_C(-p-q)) \oplus \langle \sigma \rangle
\end{equation*}
We get
\begin{equation}
  \label{eq:sym-xl}
  \Sym^2 H^0(X, L) \cong \Sym^2 H^0(C, L_C(-p-q))
  \oplus \langle \sigma \otimes \sigma \rangle
  \oplus \big( \langle \sigma \rangle \otimes H^0(C, L_C(-p-q)) \big)
\end{equation}
and for $[C, p, q]$ general the map
\begin{equation*}
  \Sym^2 H^0(X, L) \rightarrow H^0(X, L^{\otimes 2})
\end{equation*}
is injective with one-dimensional cokernel.
On the other hand, by Riemann--Roch, we have
$\dim H^0(X, L^{\otimes 2}(-p-q)) = 9$ for the space of sections vanishing at $p$ and $q$.
Comparing this with the expression \eqref{eq:sym-xl} we see that all these
sections come from $\Sym^2 H^0(X, L)$.

Now we consider the space $H^0(X, L \otimes \eta^{-1})$.  On $E$ the line bundle $L \otimes \eta^{-1}$
restricts to $\mathcal{O}_E(-1)$ and on $C$ to $L_C \otimes \eta_C^{-1}$.
Since $H^0(E, \mathcal{O}_E(-1)) = 0$ we have the identity
\begin{equation*}
  H^0(X, L \otimes \eta^{-1}) = H^0(C, L_C \otimes \eta_C^{-1} \otimes \mathcal{O}_C(-p-q))
\end{equation*}
so all sections here vanish at $p$ and $q$.  This implies that the multiplication map
\begin{equation*}
  H^0(X, L \otimes \eta) \otimes H^0(X, L \otimes \eta^{-1}) \rightarrow H^0(X, L^{\otimes 2})
\end{equation*}
factors through $H^0(X, L^{\otimes 2}(-p-q))$ and hence through the image of $\Sym^2 H^0(X, L)$.
This means that the multiplication map $\mu_{[X, \eta, L]}$ is indeed zero.
We have proved:
\begin{proposition}
  The morphism $\phi\colon \mathcal{E} \rightarrow \mathcal{F}$ of \eqref{eq:degenmorph}
  between vector bundles on $\UnivGrdTors{3}{9}{3}$
  is degenerate to order $2$ over $\Delta_0^{(1)}$.  Hence
  $[Z_1(\phi)] - 2 \delta_0^{(1)}$ is effective and therefore
  \begin{equation*}
    [\MukaiDivBar{8}{3}]^\virt - 28 \delta_0^{(1)} = 196 \lambda - 28( \delta_0' + 2\delta_0'')
    - \frac{308}{3} \delta_0^{(1)}
  \end{equation*}
  is effective as well.
\end{proposition}
\begin{theorem}
  $\RgBar{8,3}$ is of general type.
\end{theorem}
\begin{proof}
  We take the effective linear combination
  \begin{align*}
    \frac{1}{119} ( [\MukaiDivBar{8}{3}]^\virt - 28 \delta_0^{(1)} )
    + \frac{5}{17} ( [\overline{\mathcal{D}}_{8,3}] - 10 \delta_0^{(1)} )
    & \leq \frac{218}{17} \lambda - 2 ( \delta_0' + \delta_0'' ) - 4 \delta_0^{(1)}\\
    & = K_{\RgPrime{8,3}} - \frac{3}{17} \lambda
  \end{align*}
  hence $K_{\RgPrime{8,3}}$ is big.
  Now we invoke Remark 3.5 from \autocite{CEFS2013} to
  show that the same holds for $K_{\RgBar{8,3}}$.
\end{proof}

\printbibliography

\Addresses

\end{document}